\documentclass[review,1p,fleqn,12pt]{elsarticle}
\usepackage{amsmath,amsthm,amssymb}

\newtheorem{theorem}{Theorem}[section]
 \newtheorem{corollary}{Corollary}[section]

\newtheorem{lemma}{Lemma}[section]

\DeclareMathOperator{\RM}{Re}

\journal{Applied Mathematics and Computation}

\begin{document}

\begin{frontmatter}

\title{Marichev-Saigo-Maeda fractional integration operators of generalized Bessel functions}

\author[KFU]{Saiful. R. Mondal \corref{cor1}}
\address[KFU]{Department of   Mathematics an Statics ,
College of Science, King Faisal University, Al-Ahsa, Kingdom of Saudi Arabia}
\ead{saiful786@gmail.com}
\author[SAAU]{K. S. Nisar}
\address[SAAU]{Department of   Mathematics,
College of Arts and Science, Salman bin Abdul Aziz University, Wadi Al Dawaser, Kingdom of Saudi Arabia}
\ead{ksnisar1@gmail.com}
\cortext[cor1]{Corresponding author}
\begin{abstract}
Two integral operator involving the Appell's functions, or Horn's function in the kernel are considered.
Composition of such  functions with generalized Bessel functions of the first kind are expressed in term
of generalized Wright function and generalized hypergeometric series. Many special cases, including cosine and sine function are also discussed.
\end{abstract}

\end{frontmatter}

\section{Introduction}
Let $\alpha$, $\alpha'$, $\beta$, $\beta'$, $\gamma \in \mathbb{C}$ and $ x>0$, then the generalized
fractional integral operator involving the Appell's functions, or Horn's function are defined as follows:
\begin{align}\label{eqn-1-frac-opt}
\left(I_{0,{+}}^{\alpha,\alpha',\beta,\beta',\gamma}f\right)(x)&=\frac{x^{-\alpha}}{\Gamma(\gamma)}
\int_{0}^{x}(x-t)^{\gamma-1}t^{-\alpha'} F_{3}\left(\alpha,\alpha',\beta,\beta';\gamma;1-\frac{t}{x},1-\frac{x}{t}\right)f(t)dt,
\end{align}
and
\begin{align}\label{eqn-2-frac-opt}
\left(I_{0,{-}}^{\alpha,\alpha',\beta,\beta',\gamma}f\right)(x)=\frac{x^{-\alpha'}}{\Gamma(\gamma)}
\int_{x}^{\infty}(t-x)^{\gamma-1}t^{-\alpha} F_{3}\left(\alpha,\alpha',\beta,\beta';\gamma;1-\frac{t}{x},1-\frac{x}{t}\right)f(t)dt,
\end{align}
with $\RM(\gamma) >0$. The generalized fractional integral operators of the type $(\ref{eqn-1-frac-opt})$  and $(\ref{eqn-2-frac-opt})$
has been introduced by Marichev \cite{Marichev} and later extended and studied by Saigo and Maeda \cite{maeda-saigo}. This operator together known as the Marichev-Saigo-Maeda operator.

The fractional integral operator has many interesting application in various sub-fields in applicable mathematical analysis,
for example, \cite{Kim-Lee-HMSri} has applications related to certain class of complex analytic functions.
The results given in \cite{Kir1, Miller-Ross, HMSri-frac-Hfunction} can be referred for some basic results on fractional calculus.

The purpose of this work is to investigate compositions of integral
transforms ($\ref{eqn-1-frac-opt}$) and ($\ref{eqn-2-frac-opt}$) with the generalized Bessel function
of the first kind $\mathcal{W}_{p, b, c}$ defined for complex $z\in\mathbb{C}$ and $b,c,p\in\mathbb{C}$ by
\begin{equation}\label{eqn-3-bessel}
\mathcal{W}_{p, b, c}(z)=\displaystyle\sum_{k=0}^{\infty}\dfrac{(-1)^{k}c^{k}}{\Gamma(p+\frac{b+1}{2}+k) k!}\left(\dfrac{z}{2}\right)^{2k+p}.
\end{equation}
More details related to the function $\mathcal{W}_{p, b, c}$ and its particular cases can be found in
\cite{Baricz, saiful-bessel} and references therein. It is worth mentioning that, $\mathcal{W}_{p, 1, 1}= J_p$ is Bessel function
of order $p$ and $\mathcal{W}_{p, 1, -1}= I_p$ is modified Bessel function of order $p$. Also, $\mathcal{W}_{p, 2, 1}= 2 j_p/\sqrt{\pi}$ is spherical Bessel function
of order $p$ and $\mathcal{W}_{p, 2, -1}= 2 i_p/\sqrt{\pi}$ is modified spherical Bessel function of order $p$. Thus the study of the integral transform of $\mathcal{W}_{p, b, c}$ will gives far reaching results than the result in \cite{Kilbas-itsf, Purohit}.

The present article is organized as follow: in Section $\ref{sec:wright-fun}$ and Section $\ref{sec:hyper-series}$, composition of integral transforms
$(\ref{eqn-1-frac-opt})$ and $(\ref{eqn-2-frac-opt})$ with generalized Bessel function $(\ref{eqn-3-bessel})$
are given in terms of generalized Wright functions and generalized hypergeometric functions respectively.
Special  cases like $p= -b/2$ ($p= 1-b/2$) of $\mathcal{W}_{p, b, c}$ gives the composition of $(\ref{eqn-1-frac-opt})$ and $(\ref{eqn-2-frac-opt})$ with cosine  and hyperbolic cosine (sine and hyperbolic sine) functions,  are discussed  in Section $\ref{sec:trignometric}$. Some concluding remark and comparison with earlier known work are mentioned in Section $\ref{sec:conclusion}$.

Following two results given by Saigo {\it et. al.} \cite{saxena-saigo,maeda-saigo} are needed in sequel.
\begin{lemma}\label{lem-1}
Let $\alpha,\alpha',\beta,\beta',\gamma,\rho \in \mathbb{C}$ be such that
$\RM{(\gamma)}>0$ and  \[\RM{(\rho)}>\max\{0, \RM{(\alpha-\alpha'-\beta-\gamma)}, \RM{(\alpha'-\beta')}\}.\]
Then there exists the relation
\begin{align}\label{eqn-12-bessel}
\left(I_{0,+}^{\alpha,\alpha', \beta, \beta', \gamma} \; t^{\rho-1}\right)(x)
=\Gamma
\left[\begin{array}{lll}
\rho, \rho+\gamma-\alpha-\alpha'-\beta, \rho+\beta{'}-\alpha{'}\\
\rho+\beta{'}, \rho+\gamma-\alpha-\alpha{'}, \rho+\gamma-\alpha{'}-\beta
\end{array}\right] x^{\rho - \alpha - \alpha' + \gamma -1 },
\end{align}
where
\begin{align}\label{eqn-1-bessel}
\Gamma
\left[\begin{array}{lll}
a, b, c\\
d, e, f
\end{array}\right] = \frac{ \Gamma{(a)} \Gamma{(b)}  \Gamma{(c)}} { \Gamma{(d)} \Gamma{(e)}  \Gamma{(f)}}.
\end{align}
\end{lemma}
\begin{lemma}\label{lem-2}
Let $\alpha,\alpha',\beta,\beta',\gamma,\rho \in \mathbb{C}$ be such that
$\RM{(\gamma)}>0$ and  \[\RM{(\rho)}< 1+ \min\{\RM{(-\beta)}, \RM{(\alpha+\alpha'-\gamma)}, \RM{(\alpha+\beta'-\gamma)}\}.\]
Then there exists the relation
\begin{align}\label{eqn-13-bessel}
\left(I_{0,+}^{\alpha,\alpha', \beta, \beta', \gamma} \; t^{\rho-1}\right)(x)
=\Gamma\small{
\left[\begin{array}{lll}
 1-\rho-\gamma+\alpha+\alpha', 1-\rho+\alpha+\beta{'}, 1- \rho- \beta\\
1-\rho,  1-\rho+\alpha + \alpha'+\beta+ \beta'- \gamma, 1- \rho+ \alpha - \beta
\end{array}\right]} x^{\rho - \alpha - \alpha' + \gamma -1 }.
\end{align}
\end{lemma}

\section{Representations  in term of generalized Wright functions}\label{sec:wright-fun}
In this section composition of integral transforms
$(\ref{eqn-1-frac-opt})$ and $(\ref{eqn-2-frac-opt})$ with generalized Bessel function $(\ref{eqn-3-bessel})$
are given in terms of the generalized Wright hypergeometric
function ${}_p\psi_q(z)$ which is defined by the series
\begin{equation}\label{eqn-9-bessel}
{}_p\psi_q(z)={}_p\psi_q\left[\begin{array}{c}
(a_i,\alpha_i)_{1,p} \\
(b_j,\beta_j)_{1,q}
\end{array}\bigg|z\right]=\displaystyle\sum_{k=0}^{\infty}
\dfrac{\prod_{i=1}^{p}\Gamma(a_{i}+\alpha_{i}k)}{\prod_{j=1}^{q}\Gamma(b_{j}+\beta_{j}k)}
\dfrac{z^{k}}{k!}.
\end{equation}
Here $a_i, b_j\in\mathbb{C}$, and  $\alpha_i, \beta_j\in\mathbb{R}$ ($i=1,2,\ldots,p;$ $j =1,2,\ldots,q$).
Asymptotic behavior of this function for large values of argument of $z\in {\mathbb{C}}$
were studied in \cite{C. Fox} and under the condition
\begin{equation}\label{eqn-10-bessel}
\displaystyle\sum_{j=1}^{q}\beta_{j}-\displaystyle\sum_{i=1}^{p}\alpha_{i}>-1
\end{equation}
in \cite{Wright-2,Wright-3}.
Properties of this generalized Wright function were investigated in \cite{Kilbas, Kilbas-itsf, Kilbas-frac}. In particular, it was proved \cite{Kilbas} that ${}_p\psi_{q}(z)$, $z\in {\mathbb{C}}$ is an
entire function under the condition ($\ref{eqn-10-bessel}$). Interesting results related to generalized Wright functions are also given in \cite{HMSri-Wright}.

\begin{theorem}\label{thm-1}
Let $\alpha,\alpha',\beta,\beta',\gamma,\rho, p, b, c \in \mathbb{C}$
such that $\kappa :=p+ (b+1)/2 \neq 0, -1, -2,\cdots$. Suppose that
$\RM{(\gamma)}>0$ and  $\RM{(\rho+ p)}>\max\{0, \RM{(\alpha+\alpha'+\beta-\gamma)}, \RM{(\alpha'-\beta')}\}.$
Then
\begin{align}\label{eqn-1:thm-1}\nonumber
&\left(I_{0,+}^{\alpha,\alpha',\beta,\beta',\gamma} \; t^{\rho-1}\mathcal{W}_{p, b, c}(t)\right )(x)
=\frac{x^{\rho+p-\alpha-\alpha'+\gamma-1}}{2^{p}}\hspace{5in}\\
&\hspace{.5in}\times{}_3\psi_4 \left[\begin{array}{lll}
(\rho+p,2) ,(\rho+p+\gamma-\alpha-\alpha'-\beta,2)(\rho+p+\beta{'}-\alpha{'},2);\\
(\rho+p+\beta{'},2),(\rho+p+\gamma-\alpha-\alpha{'},2),\\(\rho+p+\gamma-\alpha{'}-\beta,2),(\kappa,1)
\end{array}\bigg|-\frac{cx^2}{4}\right]\end{align}
\end{theorem}
\begin{proof}
An application of integral transform $(\ref{eqn-1-frac-opt})$ to the generalized Bessel function $(\ref{eqn-3-bessel})$ leads to the formula
\begin{align}\label{eqn:2}
\left(I_{0,+}^{\alpha,\alpha',\beta,\beta',\gamma} t^{\rho-1}\mathcal{W}_{p, b, c}(t)\right )(x)
=\left(I_{0,+}^{\alpha,\alpha',\beta,\beta',\gamma}\sum\limits_{k=0}^{\infty}\frac{(-c)^k (1/2)^{2k+p}}{\Gamma(\kappa+k)k!}
t^{\rho+p+2k-1}\right)(x)
\end{align}
Now changing the order of integration and summation in  right hand side of $(\ref{eqn:2})$ yields
\begin{align}
\left(I_{0,+}^{\alpha,\alpha',\beta,\beta',\gamma} t^{\rho-1}\mathcal{W}_{p, b, c}(t)\right )(x)
=\sum\limits_{k=0}^{\infty}\frac{(-c)^k (1/2)^{2k+p}}{\Gamma(\kappa+k)k!}(I_{0,+}^{\alpha,\alpha{'},\beta,\beta',\gamma}
t^{\rho+p+2k-1})(x)
\end{align}
 Note that for all $k = 0, 1, 2, \ldots$,
\[\RM{(\rho+ p+ 2k)} \geq \RM{(\rho+ p)}>\max\{0, \RM{(\alpha-\alpha'-\beta-\gamma)}, \RM{(\alpha'-\beta')}\}.\]
Replacing  $\rho$ by $\rho+p+2k$ in Lemma $\ref{lem-1}$ and using $(\ref{eqn-12-bessel})$, we obtain
\begin{align}\label{eqn:thm-1-2} \nonumber
&\left(I_{0,+}^{\alpha,\alpha'\beta,\beta',\gamma} t^{\rho-1}\mathcal{W}_{p, b, c}(t)\right)(x)
=\frac{x^{\rho+p-\alpha-\alpha'+\gamma-1}}{2^{p}}\\
&{\tiny\times~\sum\limits_{k=0}^{\infty}
\Gamma
\left[\begin{array}{lll}
\rho+p+2k,\; \rho+p+\gamma-\alpha-\alpha{'}-\beta+2k,\; \rho+p+\beta{'}-\alpha{'}+2k\\
\rho+p+\beta{'}+2k,\; \rho+p+\gamma-\alpha-\alpha{'}+2k, \rho+p+\gamma-\alpha{'}-\beta+2k,\; \kappa+k
\end{array}\right]
\frac{1} {k!}\left(-\frac{cx^2}{4}\right)^k} .
\end{align}

Interpreting the right hand side of $(\ref{eqn:thm-1-2})$, the equality ($\ref{eqn-1:thm-1}$) can be obtained  from ($\ref{eqn-1-bessel}$) and then by using the  definition of generalized Wright function.
\end{proof}
\begin{theorem}\label{thm-2}
Let $\alpha,\alpha',\beta,\beta',\gamma,\rho, p, b, c \in \mathbb{C}$
such that $\kappa :=p+ (b+1)/2 \neq 0, -1, -2,\cdots$.
Suppose that
$\RM{(\gamma)}>0$ and  $\RM{(\rho- p)}< 1+ \min\{\RM{(-\beta)}, \RM{(\alpha+\alpha'-\gamma)}, \RM{(\alpha+\beta'-\gamma)}\}.$
Then
\begin{align*}
&\left(I_{0,-}^{\alpha,\alpha^{'},\beta,\beta',\gamma}\; t^{\rho-1}\mathcal{W}_{p, b, c}(1/t)\right )(x)=
\frac{x^{\rho-p-\alpha-\alpha^{'}+\gamma-1}}{2^{p}}\\
&\quad\times{}_3\psi_4 \tiny{\left[\begin{array}{lll}
(1-\rho+p-\gamma+\alpha+\alpha{'},2) ,(1-\rho+p+\alpha-\beta'-\gamma,2)(1-\rho+p-\beta,2);\\
(1-\rho+p,2),(1-\rho+p-\gamma+\alpha+\alpha{'}+\beta{'},2),(1-\rho+p+\alpha-\beta,2),(\kappa,k)
\end{array}\bigg|-\frac{c}{4x^2}\right]}
\end{align*}
\end{theorem}

\begin{proof}
Using $(\ref{eqn-2-frac-opt})$ and $(\ref{eqn-3-bessel})$, and then changing the order of integration and summation, which is justified under the conditions with Theorem $\ref{thm-2}$ yields
\begin{align}
\left(I_{0,-}^{\alpha,\alpha{'},\beta,\beta',\gamma}t^{\rho-1}\mathcal{W}_{p, b, c}(1/t)\right )(x)
=\sum\limits_{k=0}^{\infty}\frac{(-c)^k (1/2)^{2k+p}}{\Gamma(\kappa+k)k!}(I_{0,-}^{\alpha,\alpha{'},\beta,\beta',\gamma}
t^{\rho-p-2k-1})(x).
\end{align}
 Note that for all $k = 0, 1, 2, \ldots$,
\[\RM{(\rho-p-2k)} \leq \RM{(\rho- p)}< 1+ \min\{\RM{(-\beta)}, \RM{(\alpha+\alpha'-\gamma)}, \RM{(\alpha+\beta'-\gamma)}\}.\]
Hence replacing  $\rho$ by $\rho-p-2k$ in Lemma $\ref{lem-2}$ and using $(\ref{eqn-13-bessel})$, we obtain
\begin{align}\label{eqn:bessel-2}\nonumber
&\left(I_{0,-}^{\alpha,\alpha{'},\beta,\beta',\gamma} t^{\rho-1}\mathcal{W}_{p, b, c}(1/t)\right)(x)
=\frac{x^{\rho-p-\alpha-\alpha{'}+\gamma-1}}{2^{p}}\\
&\quad\quad\times{\tiny{\sum\limits_{k=0}^{\infty}\Gamma
\left[\begin{array}{lll}
1-\rho+p-\gamma+\alpha+\alpha{'}+2k, 1-\rho+p+\alpha+\beta{'}-\gamma+2k, 1-\rho+p-\beta+2k\\
1-\rho+p+2k, 1-\rho+p-\gamma+\alpha+\alpha{'}+\beta{'}+2k, 1-\rho+p+\alpha-\beta+2k
\end{array}\right]
\frac{1} {k!}\left(-\frac{c}{4x^2}\right)^k.}}
\end{align}
Now  $(\ref{eqn-1-bessel})$, ($\ref{eqn-9-bessel}$) and $(\ref{eqn:bessel-2})$ together imply that
\begin{align*}
&\left(I_{0,-}^{\alpha,\alpha^{'},\beta,\beta',\gamma}\; t^{\rho-1}\mathcal{W}_{p, b, c}(1/t)\right )(x)=
\frac{x^{\rho-p-\alpha-\alpha^{'}+\gamma-1}}{2^{p}}\\
&\quad\times{}_3\psi_4 \tiny{\left[\begin{array}{lll}
(1-\rho+p-\gamma+\alpha+\alpha{'},2) ,(1-\rho+p+\alpha-\beta'-\gamma,2)(1-\rho+p-\beta,2);\\
(1-\rho+p,2),(1-\rho+p-\gamma+\alpha+\alpha{'}+\beta{'},2),(1-\rho+p+\alpha-\beta,2),(p,k)
\end{array}\bigg|-\frac{c}{4x^2}\right],}
\end{align*}
and this complete the proof.
\end{proof}

\section{Representation  in term of generalized hypergeometric series}\label{sec:hyper-series}
The generalized hypergeometric function ${}_pF_q(a_1,\ldots, a_p;c_1,\ldots,c_q;z)$ is  given by
the representation
\begin{equation}\label{eqn:gn-hyper}
{}_pF_q(a_1,\ldots, a_p;c_1,\ldots,c_q;z) = \sum_{k=0}^{\infty}\dfrac{(a_1)_{k}\cdots
(a_p)_{k}}{(c_1)_{k}\cdots(c_q)_{k}(1)_{k}}z^k,\quad (z\in {\mathbb{C}}),
\end{equation}
where none of the denominator parameters is zero or a negative integer. Here
 $p$ or $q$  are allowed to be zero. The series $(\ref{eqn:gn-hyper})$
is convergent for all finite $z$ if $p \leq q$, while for $p=q+1$, it is convergent
for $|z|<1$ and divergent for $|z|>1$.

Results obtain in this section demonstrate the image formula for the generalized Bessel functions $\mathcal{W}_{p,b,c}$
under the operator $(\ref{eqn-1-frac-opt})$ and $(\ref{eqn-2-frac-opt})$ in terms of generalized hypergeometric
functions.  The well known Legendre duplication formula \cite{Erdélyi-1} given by
\begin{equation}\label{eqn-31-bessel}
\Gamma(2z)=\dfrac{2^{2z-1}}{\sqrt{\pi}}\Gamma(z)\Gamma\left(z+\frac{1}{2}\right)
\end{equation}
and
\begin{equation}\label{eqn-32-bessel}
(z)_{2k}=2^{2k}\left(\frac{z}{2}\right)_{k}\left(\frac{z+1}{2}\right)_{k}, \quad\quad (k\in\mathbb{N}_{0})
\end{equation}
are required for this purpose.
\begin{theorem}\label{thm-3}
Let $\alpha,\alpha',\beta,\beta',\gamma,\rho, p, b, c \in \mathbb{C}$
such that $\kappa :=p+ (b+1)/2 \neq  -1, -2,\cdots$. Suppose that
$\RM{(\gamma)}>0$ and  $\RM{(\rho+ p)}>\max\{0, \RM{(\alpha-\alpha'-\beta-\gamma)}, \RM{(\alpha'-\beta')}\}.$
Then there hold formula:
\begin{align*}
&\left(I_{0,+}^{\alpha,\alpha^{'},\beta,\beta',\gamma}\; t^{\rho-1}\mathcal{W}_{p, b, c}(t)\right )(x)\\&=
\frac{x^{\rho+p-1}}{2^{p}}\frac{\Gamma(\rho+p)\Gamma(\rho+p+\gamma-\alpha-\alpha{'}-\beta)}
{\Gamma(\rho+p+\beta{'})\Gamma(\rho+p+\gamma-\alpha-\alpha{'})\Gamma(\kappa)}
\frac{\Gamma(\rho+p+\beta{'}-\alpha{'})}
{\Gamma(\rho+p+\gamma-\alpha{'}-\beta)}\\
&\times{}_6F_7\small{\left[\begin{array}{lll}
\frac{\rho+p}{2},\frac{\rho+p+1}{2},\frac{\rho+p+\gamma-\alpha-\alpha{'}-\beta}{2},
\frac{\rho+p-\alpha-\alpha{'}+\beta+1}{2},
\frac{\rho+p+\beta{'}-\alpha{'}}{2},\frac{\rho+p+\beta{'}-\alpha{'}+1}{2};\\
\kappa,\frac{\rho+p+\beta{'}}{2},\frac{\rho+p+\beta{'}+1}{2},\frac{\rho+p+
\gamma-\alpha-\alpha{'}}{2},
\frac{\rho+p+\gamma-\alpha-\alpha{'}+1}{2},\\
\frac{\rho+p+\gamma-\alpha{'}-\beta}{2},\frac{\rho+p+\gamma-\alpha{'}-\beta+1}{2}
\end{array}\bigg|-\frac{cx^2}{4}\right]}
\end{align*}
\end{theorem}
\begin{proof}
It is known that $\Gamma(z+k)= \Gamma(z) (z)_k$. Thus
\begin{align*}
&\Gamma
\left[\begin{array}{lll}
\rho+p+2k,\; \rho+p+\gamma-\alpha-\alpha{'}-\beta+2k,\; \rho+p+\beta{'}-\alpha{'}+2k\\
\rho+p+\beta{'}+2k,\; \rho+p+\gamma-\alpha-\alpha{'}+2k,\; \rho+p+\gamma-\alpha{'}-\beta+2k,\; \kappa+k
\end{array}\right]\\
&= \frac{\Gamma(\rho+p)\Gamma(\rho+p+\gamma-\alpha-\alpha{'}-\beta) \Gamma(\rho+p+\beta{'}-\alpha{'})}
{\Gamma(\rho+p+\beta{'})\Gamma(\rho+p+\gamma-\alpha-\alpha{'})\Gamma(\rho+p+\gamma-\alpha{'}-\beta)\Gamma(\kappa)}\\
& \quad \quad \times \frac{ (\rho+p)_{2k} (\rho+p+\gamma-\alpha-\alpha{'}-\beta)_{2k} (\rho+p+\beta{'}-\alpha{'})_{2k}} {(\rho+p+\beta{'})_{2k} (\rho+p+\gamma-\alpha-\alpha{'})_{2k} (\rho+p+\gamma-\alpha{'}-\beta)_{2k} (\kappa)_{k}}.
\end{align*}
Now apply $(\ref{eqn-32-bessel})$ on the right hand side of above equation and then the result follows from
$(\ref{eqn:thm-1-2})$. This complete the proof.
\end{proof}

By adopting similar method next result can be obtained  from $(\ref{eqn:bessel-2})$, we omit the details.
\begin{theorem}\label{thm-4}
Let $\alpha,\alpha',\beta,\beta',\gamma,\rho, p, b, c \in \mathbb{C}$
such that $\kappa :=p+ (b+1)/2 \neq 0, -1, -2,\cdots$.
Suppose that
$\RM{(\gamma)}>0$ and  $\RM{(\rho- p)}< 1+ \min\{\RM{(-\beta)}, \RM{(\alpha+\alpha'-\gamma)}, \RM{(\alpha+\beta'-\gamma)}\}.$
Then
\begin{align*}
&\left(I_{0,+}^{\alpha,\alpha^{'},\beta,\beta',\gamma}\; t^{\rho-1}\mathcal{W}_{p, b, c}(1/t)\right )(x)\\
&=\frac{x^{\rho-p-\alpha-\alpha{'}+\gamma-1}}{2^{p}}
\frac{\Gamma(\alpha+\alpha{'}+p-\gamma-\rho+1)\Gamma(\alpha+\beta{'}-\gamma+p-\rho+1)\Gamma(-\beta+p-\rho+1)}
{\Gamma(p-\rho+1)\Gamma(\kappa)\Gamma(\alpha+\alpha{'}
+\beta{'}+p-\gamma-\rho+1)\Gamma(\alpha-\beta+p-\rho+1)}\\
&\times{}_6F_7 \left[\begin{array}{lll}
\frac{\alpha+\alpha{'}+p-\gamma-\rho+1}{2},\frac{\alpha+\alpha{'}+p+\gamma-\rho+2}{2},\frac{\alpha+\beta{'}+p-\gamma-\rho+1}{2}, \frac{\alpha+\beta{'}+p-\gamma-\rho+2}{2},\\ \frac{-\beta+p-\rho+1}{2},\frac{-\beta+p-\rho+2}{2};\\
\kappa,\frac{p-\rho+1}{2},\frac{p-\rho+2}{2},\frac{\alpha+\alpha{'}+\beta{'}
+p-\gamma-\rho+1}{2},\frac{\alpha+\alpha{'}+\beta{'}+p-\gamma-\rho+2}{2},\\
\frac{\alpha-\beta+p-\rho+1}{2},\frac{\alpha-\beta+p-\rho+2}{2}
\end{array}\bigg|-\frac{c}{4x^2}\right].
\end{align*}
\end{theorem}

\section{Fractional integration of trigonometric functions}\label{sec:trignometric}
\subsection{Cosine and hyperbolic cosine functions} For all $b\in\mathbb{C}$, if $p=-b/2$, then the generalized Bessel function $\mathcal{W}_{p, b, c}(z)$ have the form
\begin{align}\label{eqn:cos-0}
\mathcal{W}_{-\frac{b}{2}, b, c^{2}}(z)=\left(\frac{2}{z}\right)^{\frac{b}{2}}\frac{\cos c z}{\sqrt{\pi}} \quad
\text{and} \quad
\mathcal{W}_{-\frac{b}{2}, b, - c^{2}}(z)=\left(\frac{2}{z}\right)^{\frac{b}{2}}\frac{\cosh c z}{\sqrt{\pi}}.
\end{align}
Hence following results are a consequence of Theorem $\ref{thm-1}$ and Theorem $\ref{thm-2}$ respectively.
\begin{corollary}
Let $\alpha,\alpha',\beta,\beta',\gamma,\rho,  c \in \mathbb{C}$
such that
$\RM{(\gamma)}>0$ and  \[\RM{(\rho)}>\max\{0, \RM{(\alpha-\alpha'-\beta-\gamma)}, \RM{(\alpha'-\beta')}\}.\]
Then
\begin{align}\label{eqn:cos-1}\nonumber
&\left(I_{0,+}^{\alpha,\alpha{'},\beta,\beta',\gamma}\; t^{\rho-1}cos(ct)\right )(x)=\pi^{\frac{1}{2}}x^{\rho-\alpha-\alpha{'}+\gamma-1}
\\&\quad\quad \times{}_3\psi_4\small{\left[\begin{array}{lll}
(\rho,2),(\rho+\gamma-\alpha-\alpha{'}-\beta,2),(\rho+\beta{'}-\alpha{'},2);\\
(\rho+\beta{'},2),(\rho+\gamma-\alpha-\alpha{'},2),(\rho+\gamma-\alpha{'}-\beta,2),(\frac{1}{2},1)
\end{array}\bigg|-\frac{c^2x^2}{4}\right]}\hfill
\end{align}
and
\begin{align}\label{eqn:cos-2}\nonumber
&\left(I_{0,+}^{\alpha,\alpha{'},\beta,\beta',\gamma}\; t^{\rho-1}\cosh(ct)\right )(x)=\pi^{\frac{1}{2}}x^{\rho-\alpha-\alpha{'}+\gamma-1}
\\
&\quad\quad\times{}_3\psi_4\small{\left[\begin{array}{lll}
(\rho,2),(\rho+\gamma-\alpha-\alpha{'}-\beta,2),(\rho+\beta{'}-\alpha{'},2);\\
(\rho+\beta{'},2),(\rho+\gamma-\alpha-\alpha{'},2),(\rho+\gamma-\alpha{'}-\beta,2),(\frac{1}{2},1)
\end{array}\bigg|\frac{c^2x^2}{4}\right]}.
\end{align}
\end{corollary}
\begin{proof}
On setting $p=-b/2$ and replacing $c$ by $c^2$ in to ($\ref{eqn-1:thm-1}$) and using $(\ref{eqn:cos-0})$, we have
\begin{align*}
&\left(I_{0,+}^{\alpha,\alpha',\beta,\beta',\gamma} \; t^{\rho-1}\left(\frac{2}{t}\right)^{\frac{b}{2}}\frac{\cos (c t)}{\sqrt{\pi}}\right )(x)
=\frac{x^{\rho-\frac{b}{2}-\alpha-\alpha'+\gamma-1}}{2^{-\frac{b}{2}}}\hspace{5in}\\
&\hspace{.2in}\times{}_3\psi_4 {\small\left[\begin{array}{lll}
(\rho-\frac{b}{2},2) ,(\rho-\frac{b}{2}+\gamma-\alpha-\alpha'-\beta,2)(\rho-\frac{b}{2}+\beta{'}-\alpha{'},2);\\
(\rho-\frac{b}{2}+\beta{'},2),(\rho-\frac{b}{2}+\gamma-\alpha
-\alpha{'},2),(\rho-\frac{b}{2}+\gamma-\alpha{'}-\beta,2),(\frac{1}{2},1)
\end{array}\bigg|-\frac{cx^2}{4}\right]}.\end{align*}
This implies
\begin{align}\label{eqn:cos-3}\nonumber
&\left(I_{0,+}^{\alpha,\alpha',\beta,\beta',\gamma} \; t^{\rho-\frac{b}{2}-1} \cos c t\right )(x)
= \pi^{\frac{1}{2} }x^{\rho-\frac{b}{2}-\alpha-\alpha'+\gamma-1} \hspace{5in}\\
&\hspace{.2in}\times{}_3\psi_4 {\small\left[\begin{array}{lll}
(\rho-\frac{b}{2},2) ,(\rho-\frac{b}{2}+\gamma-\alpha-\alpha'-\beta,2)(\rho-\frac{b}{2}+\beta{'}-\alpha{'},2);\\
(\rho-\frac{b}{2}+\beta{'},2),(\rho-\frac{b}{2}+\gamma-\alpha-\alpha{'},2),
(\rho-\frac{b}{2}+\gamma-\alpha{'}-\beta,2),(\frac{1}{2},1)
\end{array}\bigg|-\frac{cx^2}{4}\right]}.\end{align}
The identity $(\ref{eqn:cos-1})$ follows from $(\ref{eqn:cos-3})$ by replacing $\rho$ by $\rho+\frac{b}{2}$.

Similarly, the identity  $(\ref{eqn:cos-2})$ can be obtained from ($\ref{eqn-1:thm-1}$) by  setting $p=-b/2$ and replacing $c$ by $- c^2$.
\end{proof}

\begin{corollary}

Let $\alpha,\alpha{'},\beta,\beta{'},\gamma,\rho\in\mathbb{C}$ be such that $\RM{(\gamma)}>0$, and
\[\RM{(\rho)}<\min\left\{\RM(-\beta), \RM(\alpha+\alpha{'}-\gamma), \RM(\alpha+\beta{'}-\gamma)\right\}.\]
Then
\begin{align*}
&\left(I_{0,+}^{\alpha,\alpha^{'},\beta,\beta',\gamma}\; t^{\rho-1}\cos\left(\dfrac{c}{t}\right)\right )(x)=\pi^{\frac{1}{2}}x^{\rho-\alpha-\alpha{'}+\gamma}
\\&\quad\quad\times{}_3\psi_4 \left[\begin{array}{lll}
(\rho-\gamma+\alpha+\alpha{'},2),(-\rho+\alpha+\beta{'}-\gamma,2),(-\rho-\beta,2);\\
(-\rho,2),(-\rho-\gamma+\alpha+\alpha{'}+\beta{'},2),(-\rho+\alpha-\beta,2),(\frac{1}{2},1)
\end{array}\bigg|-\frac{c^2}{4x^2}\right]\end{align*}
and
\begin{align*}
&\left(I_{0,+}^{\alpha,\alpha^{'},\beta,\beta',\gamma}\; t^{\rho}\cosh\left(\dfrac{c}{t}\right)\right )(x)=\pi^{\frac{1}{2}}x^{\rho-\alpha-\alpha{'}+\gamma}
\\&\quad\quad\quad\times{}_3\psi_4 \left[\begin{array}{lll}
(\rho-\gamma+\alpha+\alpha{'},2),(-\rho+\alpha+\beta{'}-\gamma,2),(-\rho-\beta,2);\\
(-\rho,2),(-\rho-\gamma+\alpha+\alpha{'}+\beta{'},2),(-\rho+\alpha-\beta,2),(\frac{1}{2},1)
\end{array}\bigg|\frac{c^2}{4x^2}\right].\end{align*}
\end{corollary}
The next statements shows that the image formulas for cosine and hyperbolic cosine under Saigo-Maeda fractional integral operators can also be represent in terms of the generalized hypergeometric series. This result follows
from Theorem $\ref{thm-3}$ and Theorem $\ref{thm-4}$ with taking $p=-{b}/{2}$ and replacing $c$ by $c^2$ or $-c^2$ respectively.

\begin{corollary}
Let $\alpha,\alpha',\beta,\beta',\gamma,\rho, c \in \mathbb{C}$. Suppose that
$\RM{(\gamma)}>0$ and  \[\RM{(\rho)}>\max\{0, \RM{(\alpha-\alpha'-\beta-\gamma)}, \RM{(\alpha'-\beta')}\}.\]
Then there hold formula:
\begin{align*}
&\left(I_{0,+}^{\alpha,\alpha^{'},\beta,\beta',\gamma}\; t^{\rho-1}\cos(ct)\right )(x)\\
&=\frac{\Gamma(\rho)\Gamma(\rho+\gamma-\alpha-\alpha{'}-\beta)\Gamma(\rho+\beta{'}-\alpha{'})}{\Gamma(\rho+\beta{'})
\Gamma(\rho+\gamma-\alpha-\alpha{'})\Gamma(\rho+\gamma-\alpha{'}-\beta)}x^{\rho-\alpha-\alpha{'}+\gamma-1}\\
&\hspace{.4in}\times{}_6F_7 \left[\begin{array}{lll}
\frac{\rho}{2},\frac{\rho+1}{2},\frac{\rho+\gamma-\alpha-\alpha{'}-\beta}{2},\frac{\rho+\gamma-\alpha-\alpha{'}+\beta+1}{2},
\frac{\rho+\beta{'}-\alpha{'}}{2},\frac{\rho+\beta{'}-\alpha{'}+1}{2};\\
\frac{1}{2},\frac{\rho+\beta{'}}{2},\frac{\rho+\beta{'}+1}{2},\frac{\rho+\gamma-\alpha-\alpha{'}}{2},
\frac{\rho+\gamma-\alpha-\alpha{'}+1}{2},\frac{\rho+\gamma-\alpha{'}-\beta}{2},\frac{\rho+\gamma-\alpha{'}-\beta+1}{2}
\end{array}\bigg|-\frac{c^2x^2}{4}\right] \end{align*}
and
\begin{align*}
&\left(I_{0,+}^{\alpha,\alpha^{'},\beta,\beta',\gamma}\; t^{\rho-1}\cosh(ct)\right )(x)\\
&=\frac{\Gamma(\rho)\Gamma(\rho+\gamma-\alpha-\alpha{'}-\beta)\Gamma(\rho+\beta{'}-\alpha{'})}{\Gamma(\rho+\beta{'})
\Gamma(\rho+\gamma-\alpha-\alpha{'})\Gamma(\rho+\gamma-\alpha{'}-\beta)}x^{\rho-\alpha-\alpha{'}+\gamma-1}\\
&\hspace{.5in}\times{}_6F_7 \left[\begin{array}{lll}
\frac{\rho}{2},\frac{\rho+1}{2},\frac{\rho+\gamma-\alpha-\alpha{'}-\beta}{2},\frac{\rho+\gamma-\alpha-\alpha{'}+\beta+1}{2},
\frac{\rho+\beta{'}-\alpha{'}}{2},\frac{\rho+\beta{'}-\alpha{'}+1}{2};\\
\frac{1}{2},\frac{\rho+\beta{'}}{2},\frac{\rho+\beta{'}+1}{2},\frac{\rho+\gamma-\alpha-\alpha{'}}{2},
\frac{\rho+\gamma-\alpha-\alpha{'}+1}{2},\frac{\rho+\gamma-\alpha{'}-\beta}{2},\frac{\rho+\gamma-\alpha{'}-\beta+1}{2}
\end{array}\bigg|\frac{c^2x^2}{4}\right].\end{align*}
\end{corollary}
\begin{corollary}
Let $\alpha,\alpha',\beta,\beta',\gamma,\rho, c \in \mathbb{C}$.
Suppose that
$\RM{(\gamma)}>0$ and
\[\RM{(\rho)}< 1+ \min\{\RM{(-\beta)}, \RM{(\alpha+\alpha'-\gamma)}, \RM{(\alpha+\beta'-\gamma)}\}.\]
Then
\begin{align*}
&\left(I_{0,-}^{\alpha,\alpha{'},\beta,\beta',\gamma}\; t^{\rho}\cos\left(\frac{c}{t}\right)\right )(x)\\
&=\frac{\Gamma(\alpha+\alpha{'}-\gamma-\rho)\Gamma(\alpha+\beta{'}-\gamma-\rho) \Gamma(-\beta-\rho)}
{\Gamma(-\rho)\Gamma(\alpha+\alpha{'}+\beta{'}-\gamma-\rho)\Gamma(\alpha-\beta-\rho)}
x^{\rho-\alpha-\alpha{'}+\gamma}\\
&\hspace{.5in}\times{}_6F_7\left[\begin{array}{lll}
\frac{\alpha+\alpha{'}-\rho}{2},\frac{\alpha+\alpha{'}-\gamma-\rho+1}{2},\frac{\alpha+\beta{'}-\gamma-\rho}{2},\frac{\alpha+\beta{'}-\gamma-\rho+1}{2},\frac{-\beta-\rho}{2},\frac{-\beta-\rho+1}{2};\\
\frac{1}{2},-\frac{\rho}{2},-\frac{\rho+1}{2},\frac{\alpha+\alpha{'}+\beta{'}-\gamma-\rho}{2},\frac{\alpha+\alpha{'}+\beta{'}-\gamma-\rho+1}{2},\frac{\alpha-\beta-\rho}{2},\frac{\alpha-\beta-\rho+1}{2},
\end{array}\bigg|-\frac{c^2}{4x^2}\right]
\end{align*}
and
\begin{align*}
&\left(I_{0,-}^{\alpha,\alpha^{'},\beta,\beta',\gamma}\; t^{\rho}\cosh\left(\frac{c}{t}\right)\right )(x)\\
&=\frac{\Gamma(\alpha+\alpha{'}-\gamma-\rho)\Gamma(\alpha+\beta{'}-\gamma-\rho)}
{\Gamma(-\rho)\Gamma(\alpha+\alpha{'}+\beta{'}-\gamma-\rho)}\frac{\Gamma(-\beta-\rho)}{\Gamma(\alpha-\beta-\rho)}
x^{\rho-\alpha-\alpha{'}+\gamma}\\
&\hspace{.5in}\times{}_6F_7\left[\begin{array}{lll}
\frac{\alpha+\alpha{'}-\rho}{2},\frac{\alpha+\alpha{'}-\gamma-\rho+1}{2},\frac{\alpha+\beta{'}
-\gamma-\rho}{2},\frac{\alpha+\beta{'}-\gamma-\rho+1}{2},\frac{-\beta-\rho}{2},\frac{-\beta-\rho+1}{2};\\
\frac{1}{2},-\frac{\rho}{2},-\frac{\rho+1}{2},\frac{\alpha+\alpha{'}+\beta{'}-\gamma-\rho}{2},
\frac{\alpha+\alpha{'}+\beta{'}-\gamma-\rho+1}{2},\frac{\alpha-\beta-\rho}{2},\frac{\alpha-\beta-\rho+1}{2}
\end{array}\bigg|\frac{c^2}{4x^2}\right].
\end{align*}
\end{corollary}

\subsection{Sine and hyperbolic sine functions}
For all $b\in\mathbb{C}$, if $p=1-b/2$, then the generalized Bessel function $\mathcal{W}_{p, b, c}(z)$ have the form
\begin{align}\label{eqn:cos-0}
\mathcal{W}_{1-\frac{b}{2}, b, c^{2}}(z)=\left(\frac{2}{z}\right)^{\frac{b}{2}}\frac{\sin( c z)}{\sqrt{\pi}} \quad
\text{and} \quad
\mathcal{W}_{1-\frac{b}{2}, b, - c^{2}}(z)=\left(\frac{2}{z}\right)^{\frac{b}{2}}\frac{\sinh( c z)}{\sqrt{\pi}}.
\end{align}
Thus, the composition of Saigo-Maeda fractional integral operators with sine and hyperbolic sine functions
can be obtained from Theorem $\ref{thm-1}$ and Theorem $\ref{thm-2}$ respectively.
\begin{corollary}
Let $\alpha,\alpha',\beta,\beta',\gamma,\rho,  c \in \mathbb{C}$
such that
$\RM{(\gamma)}>0$ and  \[\RM{(\rho)}>\max\{0, \RM{(\alpha-\alpha'-\beta-\gamma)}, \RM{(\alpha'-\beta')}\}.\]
Then
\begin{align}\label{eqn:cos-1}\nonumber
&\left(I_{0,+}^{\alpha,\alpha{'},\beta,\beta',\gamma}\; t^{\rho-1} \sin(ct)\right )(x)=\pi^{\frac{1}{2}}x^{\rho-\alpha-\alpha{'}+\gamma-1}\\
&\quad\quad\times{}_3\psi_4\small{\left[\begin{array}{lll}
(\rho,2),(\rho+\gamma-\alpha-\alpha{'}-\beta,2),(\rho+\beta{'}-\alpha{'},2);\\
(\rho+\beta{'},2),(\rho+\gamma-\alpha-\alpha{'},2),(\rho+\gamma-\alpha{'}-\beta,2),(\frac{3}{2},1)
\end{array}\bigg|-\frac{c^2x^2}{4}\right]}
\end{align}
and
\begin{align}\label{eqn:cos-2}\nonumber
&\left(I_{0,+}^{\alpha,\alpha{'},\beta,\beta',\gamma}\; t^{\rho-1}\sinh(ct)\right )(x)=\pi^{\frac{1}{2}}x^{\rho-\alpha-\alpha{'}+\gamma-1}\\
&\quad \quad\times{}_3\psi_4\small{\left[\begin{array}{lll}
(\rho,2),(\rho+\gamma-\alpha-\alpha{'}-\beta,2),(\rho+\beta{'}-\alpha{'},2);\\
(\rho+\beta{'},2),(\rho+\gamma-\alpha-\alpha{'},2),(\rho+\gamma-\alpha{'}-\beta,2),(\frac{3}{2},1)
\end{array}\bigg|\frac{c^2x^2}{4}\right]}.
\end{align}
\end{corollary}
%\begin{proof}
%On setting $p=1-b/2$ and replacing $c$ by $c^2$ in to ($\ref{eqn-1:thm-1}$) and using $(\ref{eqn:cos-0})$, we have
%\begin{align*}
%&\left(I_{0,+}^{\alpha,\alpha',\beta,\beta',\gamma} \; t^{\rho-1}\left(\frac{2}{t}\right)^{\frac{b}{2}}\frac{\sin (c t)}{\sqrt{\pi}}\right )(x)
%=\frac{x^{\rho-\frac{b}{2}-\alpha-\alpha'+\gamma-1}}{2^{-\frac{b}{2}}}\hspace{5in}\\
%&\hspace{.5in}\times{}_3\psi_4 {\small\left[\begin{array}{lll}
%(\rho-\frac{b}{2},2) ,(\rho-\frac{b}{2}+\gamma-\alpha-\alpha'-\beta,2)(\rho-\frac{b}{2}+\beta{'}-\alpha{'},2);\\
%(\rho-\frac{b}{2}+\beta{'},2),(\rho-\frac{b}{2}+\gamma-\alpha-\alpha{'},2),(\rho-\frac{b}{2}+\gamma-\alpha{'}-\beta,2),(3/2,1)
%\end{array}\bigg|-\frac{c^2x^2}{4}\right]}.\end{align*}
%This implies
%\begin{align}\label{eqn:cos-3}\nonumber
%&\left(I_{0,+}^{\alpha,\alpha',\beta,\beta',\gamma} \; t^{\rho-\frac{b}{2}-1} \sin c t\right )(x)
%= \pi^{\frac{1}{2} }x^{\rho-\frac{b}{2}-\alpha-\alpha'+\gamma-1} \hspace{5in}\\
%&\hspace{.5in}\times{}_3\psi_4 {\small\left[\begin{array}{lll}
%(\rho-\frac{b}{2},2) ,(\rho-\frac{b}{2}+\gamma-\alpha-\alpha'-\beta,2)(\rho-\frac{b}{2}+\beta{'}-\alpha{'},2);\\
%(\rho-\frac{b}{2}+\beta{'},2),(\rho-\frac{b}{2}+\gamma-\alpha-\alpha{'},2),(\rho-\frac{b}{2}+\gamma-\alpha{'}-\beta,2),(3/2,1)
%\end{array}\bigg|-\frac{c^2x^2}{4}\right]}.\end{align}
%The identity $(\ref{eqn:cos-1})$ follows from $(\ref{eqn:cos-3})$ by replacing $\rho$ by $\rho+\frac{b}{2}$.
%
%Similarly, the identity  $(\ref{eqn:cos-2})$ can be obtained from ($\ref{eqn-1:thm-1}$) by setting setting $p=-b/2$ and replacing $c$ by $- c^2$.
%\end{proof}

The next result follows from Theorem $\ref{thm-2}$ by setting $p = - b/2$
and replacing $c$ by $c^2$ or $- c^2$ respectively.

\begin{corollary}
Let $\alpha,\alpha{'},\beta,\beta{'},\gamma,\rho\in\mathbb{C}$ be such that $\RM{(\gamma)}>0$, and
\[\RM{(\rho)}<\min\left\{\RM(-\beta), \RM(\alpha+\alpha{'}-\gamma), \RM(\alpha+\beta{'}-\gamma)\right\}.\]
Then
\begin{align*}
&\left(I_{0,+}^{\alpha,\alpha^{'},\beta,\beta',\gamma}\; t^{\rho-1}\sin\left(\dfrac{c}{t}\right)\right )(x)=\pi^{\frac{1}{2}}x^{\rho-\alpha-\alpha{'}+\gamma}\\&\quad \quad
\times{}_3\psi_4 \left[\begin{array}{lll}
(\rho-\gamma+\alpha+\alpha{'},2),(-\rho+\alpha+\beta{'}-\gamma,2),(-\rho-\beta,2);\\
(-\rho,2),(-\rho-\gamma+\alpha+\alpha{'}+\beta{'},2),(-\rho+\alpha-\beta,2),(\frac{3}{2},1)
\end{array}\bigg|-\frac{c^2}{4x^2}\right]\end{align*}
and
\begin{align*}
&\left(I_{0,+}^{\alpha,\alpha^{'},\beta,\beta',\gamma}\; t^{\rho}\sinh\left(\dfrac{c}{t}\right)\right )(x)=\pi^{\frac{1}{2}}x^{\rho-\alpha-\alpha{'}+\gamma}
\\&\quad\quad\times{}_3\psi_4 \left[\begin{array}{lll}
(\rho-\gamma+\alpha+\alpha{'},2),(-\rho+\alpha+\beta{'}-\gamma,2),(-\rho-\beta,2);\\
(-\rho,2),(-\rho-\gamma+\alpha+\alpha{'}+\beta{'},2),(-\rho+\alpha-\beta,2),(\frac{3}{2},1)
\end{array}\bigg|\frac{c^2}{4x^2}\right].\end{align*}
\end{corollary}
Following result can be obtained
from Theorem $\ref{thm-3}$ and Theorem $\ref{thm-4}$ with taking $p=-{b}/{2}$ and replacing $c$ by $c^2$ or $-c^2$ respectively.

\begin{corollary}
Let $\alpha,\alpha',\beta,\beta',\gamma,\rho, c \in \mathbb{C}$. Suppose that
$\RM{(\gamma)}>0$ and  \[\RM{(\rho)}>\max\{0, \RM{(\alpha-\alpha'-\beta-\gamma)}, \RM{(\alpha'-\beta')}\}.\]
Then there hold formula:
\begin{align*}
&\left(I_{0,+}^{\alpha,\alpha^{'},\beta,\beta',\gamma}\; t^{\rho-1}\sin(ct)\right )(x)
\\&=\frac{\Gamma(\rho)\Gamma(\rho+\gamma-\alpha-\alpha{'}-\beta)\Gamma(\rho+\beta{'}-\alpha{'})}{\Gamma(\rho+\beta{'})
\Gamma(\rho+\gamma-\alpha-\alpha{'})\Gamma(\rho+\gamma-\alpha{'}-\beta)}x^{\rho-\alpha-\alpha{'}+\gamma-1}\\
&\hspace{.5in}\times{}_6F_7 \left[\begin{array}{lll}
\frac{\rho}{2},\frac{\rho+1}{2},\frac{\rho+\gamma-\alpha-\alpha{'}-\beta}{2},\frac{\rho+\gamma-\alpha-\alpha{'}+\beta+1}{2},
\frac{\rho+\beta{'}-\alpha{'}}{2},\frac{\rho+\beta{'}-\alpha{'}+1}{2};\\
\frac{3}{2},\frac{\rho+\beta{'}}{2},\frac{\rho+\beta{'}+1}{2},\frac{\rho+\gamma-\alpha-\alpha{'}}{2},
\frac{\rho+\gamma-\alpha-\alpha{'}+1}{2},\frac{\rho+\gamma-\alpha{'}-\beta}{2},\frac{\rho+\gamma-\alpha{'}-\beta+1}{2}
\end{array}\bigg|-\frac{c^2x^2}{4}\right], \end{align*}
and
\begin{align*}
&\left(I_{0,+}^{\alpha,\alpha^{'},\beta,\beta',\gamma}\; t^{\rho-1}\sinh(ct)\right )(x)\\
&=\frac{\Gamma(\rho)\Gamma(\rho+\gamma-\alpha-\alpha{'}-\beta)\Gamma(\rho+\beta{'}-\alpha{'})}{\Gamma(\rho+\beta{'})
\Gamma(\rho+\gamma-\alpha-\alpha{'})\Gamma(\rho+\gamma-\alpha{'}-\beta)}x^{\rho-\alpha-\alpha{'}+\gamma-1}\\
&\hspace{.5in}\times{}_6F_7 \left[\begin{array}{lll}
\frac{\rho}{2},\frac{\rho+1}{2},\frac{\rho+\gamma-\alpha-\alpha{'}-\beta}{2},\frac{\rho+\gamma-\alpha-\alpha{'}+\beta+1}{2},
\frac{\rho+\beta{'}-\alpha{'}}{2},\frac{\rho+\beta{'}-\alpha{'}+1}{2};\\
\frac{3}{2},\frac{\rho+\beta{'}}{2},\frac{\rho+\beta{'}+1}{2},\frac{\rho+\gamma-\alpha-\alpha{'}}{2},
\frac{\rho+\gamma-\alpha-\alpha{'}+1}{2},\frac{\rho+\gamma-\alpha{'}-\beta}{2},\frac{\rho+\gamma-\alpha{'}-\beta+1}{2}
\end{array}\bigg|\frac{c^2x^2}{4}\right] .\end{align*}
\end{corollary}
\begin{corollary}
Let $\alpha,\alpha',\beta,\beta',\gamma,\rho, c \in \mathbb{C}$.
Suppose that
$\RM{(\gamma)}>0$ and  $\RM{(\rho)}< 1+ \min\{\RM{(-\beta)}, \RM{(\alpha+\alpha'-\gamma)}, \RM{(\alpha+\beta'-\gamma)}\}.$
Then
\begin{align*}
&\left(I_{0,-}^{\alpha,\alpha^{'},\beta,\beta',\gamma}\; t^{\rho}\sin\left(\frac{c}{t}\right)\right )(x)\\&=\frac{\Gamma(\alpha+\alpha{'}-\gamma-\rho)\Gamma(\alpha+\beta{'}-\gamma-\rho)}
{\Gamma(-\rho)\Gamma(\alpha+\alpha{'}+\beta{'}-\gamma-\rho)}\frac{\Gamma(-\beta-\rho)}{\Gamma(\alpha-\beta-\rho)}
x^{\rho-\alpha-\alpha{'}+\gamma}\\
&\hspace{.5in}\times{}_6F_7\left[\begin{array}{lll}
\frac{\alpha+\alpha{'}-\rho}{2},\frac{\alpha+\alpha{'}-\gamma-\rho+1}{2},\frac{\alpha+\beta{'}-\gamma-\rho}{2},
\frac{\alpha+\beta{'}-\gamma-\rho+1}{2},\frac{-\beta-\rho}{2},\frac{-\beta-\rho+1}{2};\\
\frac{3}{2},-\frac{\rho}{2},-\frac{\rho+1}{2},\frac{\alpha+\alpha{'}+\beta{'}-\gamma-\rho}{2},
\frac{\alpha+\alpha{'}+\beta{'}-\gamma-\rho+1}{2},\frac{\alpha-\beta-\rho}{2},\frac{\alpha-\beta-\rho+1}{2},
\end{array}\bigg|-\frac{c^2}{4x^2}\right]
\end{align*}
and
\begin{align*}
&\left(I_{0,-}^{\alpha,\alpha^{'},\beta,\beta',\gamma}\; t^{\rho}\sinh\left(\frac{c}{t}\right)\right )(x)\\&=\frac{\Gamma(\alpha+\alpha{'}-\gamma-\rho)\Gamma(\alpha+\beta{'}-\gamma-\rho)}
{\Gamma(-\rho)\Gamma(\alpha+\alpha{'}+\beta{'}-\gamma-\rho)}\frac{\Gamma(-\beta-\rho)}{\Gamma(\alpha-\beta-\rho)}
x^{\rho-\alpha-\alpha{'}+\gamma}\\
&\hspace{.5in}\times{}_6F_7\left[\begin{array}{lll}
\frac{\alpha+\alpha{'}-\rho}{2},\frac{\alpha+\alpha{'}-\gamma-\rho+1}{2},\frac{\alpha+\beta{'}
-\gamma-\rho}{2},\frac{\alpha+\beta{'}-\gamma-\rho+1}{2},\frac{-\beta-\rho}{2},\frac{-\beta-\rho+1}{2};\\
\frac{3}{2},-\frac{\rho}{2},-\frac{\rho+1}{2},\frac{\alpha+\alpha{'}+\beta{'}-\gamma
-\rho}{2},\frac{\alpha+\alpha{'}+\beta{'}-\gamma-\rho+1}{2},\frac{\alpha-\beta-\rho}{2},\frac{\alpha-\beta-\rho+1}{2}
\end{array}\bigg|\frac{c^2}{4x^2}\right].
\end{align*}
\end{corollary}

\section{Concluding Observations}\label{sec:conclusion}
In this section some consequence of the main result derived in previous sections
are given in details. Also comparison with other known results from literature are listed.
\begin{enumerate}
\item[1.] We remark that  all the results given by Purohit {\it et. al.} \cite{Purohit}
are follows from the results derived in this article by setting $b=1=c$.

\item[2.] The results in Section $\ref{sec:wright-fun}$ and Section $\ref{sec:hyper-series}$
also provide the Marichev-Saigo-Maeda fractional integration of modified Bessel function and spherical
Bessel functions.

\item[3.] Set $\alpha'=0$ in the operator $(\ref{eqn-1-frac-opt})$ and  $(\ref{eqn-2-frac-opt})$. Then  due to identities given by Saxena and Saigo \cite[ P.93]{saxena-saigo}, it follows that
    \begin{align*}
    \left(I_{0, +}^{\alpha,  0, \beta, \beta', \gamma} f\right)(x) &=   \left(I_{0, x}^{\gamma,\alpha-\gamma, -\beta} f\right)(x)\\
     \left(I_{0, -}^{\alpha,  0, \beta, \beta', \gamma} f\right)(x) &=   \left(I_{x, \infty}^{\gamma,\alpha-\gamma, -\beta} f\right)(x)
    \end{align*}
The generalized integral transforms that appear in the right hand side of above equations is due to Saigo \cite{M Saigo}
and defined as follows:
\begin{equation}\label{eqn-1-sagio}
\left(I_{0_{+}}^{\alpha,\beta,\eta}f\right)(x)=\frac{x^{-\alpha-\beta}}{\Gamma(\alpha)}\displaystyle
\int_{0}^{x}(x-t)^{\alpha-1}{}_2F_{1}
\left(\alpha+\beta,-\eta;\alpha;1-\frac{t}{x}\right)f(t)dt
\end{equation}
and
\begin{equation}\label{eqn-2-sagio}
\left(I_{-}^{\alpha,\beta,\eta}f\right)(x)=\frac{1}{\Gamma(\alpha)}\displaystyle\int_{x}^{\infty}
(t-x)^{\alpha-1}t^{-\alpha-\beta}{}_2F_{1}
\left(\alpha+\beta,-\eta;\alpha;1-\frac{x}{t}\right)f(t)dt,
\end{equation}
where $\Gamma(\alpha)$ is the Euler gamma function \cite{Rainville},  and
${}_2F_{1}(a,b;c;x)$ is the Gauss hypergeometric function.

The above fact help us to conclude that all the results given in \cite{Kilbas-itsf} and \cite{Pradeep-Saiful} can also
be  obtained from the results in this article by setting $\alpha'=0$.

\item[4.] Note that Riemann-Liouville, Weyl and Erd\'elyi-Kober fractional calculus \cite{Samko} are
special case of Saigo's operator $(\ref{eqn-1-sagio})$ and $(\ref{eqn-2-sagio})$. Thus this article is also useful
to derive ceratin composition formula involving Riemann-Liouville, Weyl and Erdl\'eyi-Kober fractional calculus and
Bessel, modified Bessel, and spherical Bessel function of first kind.
\end{enumerate}

\end{document}